\documentclass[a4paper,oneside,11pt]{article}
\usepackage[english]{babel}
\usepackage[latin1]{inputenc}
\usepackage{indentfirst}
\usepackage{amssymb}
\usepackage{amsmath}
\usepackage{amsthm}
\usepackage{amscd}

\theoremstyle{plain}
 \newtheoremstyle{miestilo}{12pt}{\topsep}{\itshape}{}{\bf}{}{ }{}
\swapnumbers
 \theoremstyle{miestilo}
\newtheorem{theorem}[subsection]{Theorem.}
\newtheorem{proposition}[subsection]{Proposition.}
\newtheorem{lemma}[subsection]{Lemma.}
\newtheorem{corollary}[subsection]{Corollary.}
 \newtheoremstyle{misnotas}{12pt}{12pt}{}{}{\bf}{}{ }{\remark}
 \theoremstyle{misnotas}
  \newtheorem{definition}[subsection]{\ {\bf Definition.}}
 \newtheorem{remark}[subsection]{\ {\bf Remark.}}
  \newtheorem{example}[subsection]{\ {\bf Example.}}

 \allowdisplaybreaks[2]

\begin{document}

\flushbottom

 \medskip
  {\large\centerline{\bf Continuous  Evolution Algebras }}

 \medskip

\centerline{ Fernando Montaner \footnote{Partially supported   by grant MTM2017-83506-C2-1-P (AEI/FEDER, UE), and by
 grant E22 20R (Gobierno de Arag\'on, Grupo de referencia \'Algebra y Geometr\'{i}a, cofunded by FEDER 2014-2020 Construyendo Europa desde Arag\'{o}n).}} \centerline{{\sl Departamento de Matem\'{a}ticas,
Universidad de Zaragoza}} \centerline{{\sl 50009 Zaragoza,
Spain}} \centerline{E-mail: fmontane@unizar.es} \centerline{and}
\centerline{ Irene Paniello \footnote{Partially supported   by grant MTM2017-83506-C2-1-P (AEI/FEDER, UE).}} \centerline{{\sl
Departamento de Estad\'{\i}stica, Inform\'{a}tica y  Matem\'{a}ticas, Universidad
P\'{u}blica de Navarra}}
 \centerline{{\sl 31006 Pamplona, Spain}}
\centerline{E-mail: irene.paniello@unavarra.es}

 \medskip

\begin{abstract}
We formulate the notion of  continuous evolution algebra in terms of differentiable matrix-valued functions, to then study those such algebras arising as solutions of ODE problems. Given their dependence on natural bases, matrix Lie groups provide a suitable framework where considering   time-variant evolution algebras.
  We conclude by broadening our approach by considering continuous evolution algebras stemming as flow lines   on  matrix Lie groups.
 \end{abstract}

\noindent{\bf  Keywords:}   Continuous evolution algebra, Matrix Lie group, Flow, Vector field.

\noindent{\bf 2010 Mathematics Subject Classification:} 17D92,  17A01, 92D25.

\section{Introduction}

Evolution algebras were introduced in \cite{Tian EA,Tian-Vojtechovsky}  aimed to study dynamics in non-Mendelian genetic systems, and soon were connected to a broad number of other areas of research. One of their most fruitful interactions stems from homogeneous discrete-time (HDT) Markov chains. This connection,
  already settled in \cite{Tian EA}, has been recently revisited in \cite{paniello-Markov}, where the HDT case was generalized to a more general continuous-time framework.

Characterized as nonassociative algebras admitting (natural) bases, for which the only non-vanishing products arise from the squares of the natural basis elements, evolution algebras have square (and  not cubic) structure matrices. A real evolution algebra is then called Markov   when a natural basis exists yielding a nonnegative row stochastic (i.e. Markov) structure matrix   \cite[Remark~2.1]{paniello-Markov}. The passage from time-invariant to time-variant Markov evolution algebras was considered in \cite{paniello-Markov}, and defined in terms of standard stochastic semigroups.

 Given a finite dimensional (real) vector space ${\cal E}$ with basis ${\cal B}=\{e_1,\dots, e_n\}$, a family ${\cal E}(t)=\{{\cal E}_t=({\cal E}, m(t))\}_{t\geq0}$ of evolution algebras with multiplication:
$$  m(t)(e_i\otimes e_j)=e_i\cdot_t e_j=\left\{
      \begin{array}{ll}
         \sum_{k=1}^n a_{ik}(t)e_k, & \hbox{$i=j=1,\ldots,n$;} \\
        0, & \hbox{otherwise;}
      \end{array}
    \right.
$$
is a {\sl continuous time Markov evolution algebra} (CT-Markov EA) if the structure matrices $\{\mathbf{A}(t)\}_{t\geq0}$ (of each ${\cal E}_t$ w.r.t. ${\cal B}$) define a {\sl standard stochastic semigroup} on the finite index set $\Lambda=\{1,\ldots,n\}$. Then,  for each $t,s\geq0$:
 \begin{enumerate}
 \item[(i)] $\mathbf{A}(t)$ is a Markov matrix.
  \item[(ii)] $\mathbf{A}(0)=\mathbf{I}_n$.
   \item[(iii)] $\mathbf{A}(t+s)= \mathbf{A}(t )\mathbf{A}(s) $  (Chapman-Kolmogorov equation or semigroup property).
    \item[(iv)] $\lim_{t\to 0^+}\mathbf{A}(t)=\mathbf{A}(0)=\mathbf{I}_n$, componentwise (standard property).
 \end{enumerate}
Finite state standard stochastic semigroups are solutions of Backward and Forward Kolmogorov differential equations: (B) $ \mathbf{A}'(t)=\mathbf{Q}\mathbf{A}(t)$ and (F) $\mathbf{A}'(t)=\mathbf{A}(t)\mathbf{Q}$,
with initial condition $\mathbf{A}(0)=\mathbf{I}_n$. The unique solution is $\mathbf{A}(t)=e^{t\mathbf{Q}}$, for a {\sl rate matrix} or {\sl Markov generator} $\mathbf{Q}$, a matrix with nonnegative off-diagonal entries and row  sums equal to zero. It also holds  $\mathbf{A}'(0)=\mathbf{Q}$.
Moreover, since $\det(\mathbf{A}(t))= e^{tr(t\mathbf{Q})}$, matrices in finite standard stochastic semigroups are nonsingular matrices   in the stochastic group $S(n,\mathbb{R})$ \cite{poole}.

Aimed by Tian's outlined notion of continuous evolution algebra  \cite[Subsection~6.2.4]{Tian EA}, in the second section we formulate continuous evolution algebras to be a family of finite dimensional algebras, defined on the same  $\mathbb{K}$-vector space, endowed
   with a natural basis for which the corresponding structure matrices define   differentiable matrix-valued functions, that is, a differentiable curves in $M_n(\mathbb{K})$, for a field  $\mathbb{K}$ (usually $\mathbb{R}$  or $\mathbb{C}$).
   It is  then straightforward to consider those algebras arising as solutions of first order matrix ODE  problems $\mathbf{A}'(t)= \mathbf{A}(t)\mathbf{X}$ (equivalently $\mathbf{A}'(t)=\mathbf{X}\mathbf{A}(t)$) with initial condition $\mathbf{A}(0)=\mathbf{I}_n$.

In the third section matrix Lie groups provide a suitable framework where considering continuous evolution algebras.  One-paremeter subgroups  provide immediate examples of $G$-continuous evolution algebras, that is, continuous evolution algebras with structure matrices into a matrix Lie group $G$. These algebras are then characterized by their  velocity vectors  (structure matrix derivative at $t=0$) and an initial condition (structure matrix $\mathbf{A}(0)=\mathbf{I}_n$ at $t=0$). This brings into
 the Lie algebra $\mathfrak{g}$ of the matrix Lie group $G$ or equivalently the tangent space $T_{\mathbf{I}_n}(G)$ of $G$ at $\mathbf{I}_n$.
Different  initial conditions lead us to consider the tangent bundle $T(G)$ of $G$.

Continuous time Markov evolution algebras \cite{paniello-Markov} do fit into the previous scheme, as it is shown in section four,
by taking into account the properties of the exponential matrix series.

In the fifth section we show how $G$-continuous evolution algebras, resulting
 in section~3, make way towards
  continuous evolution algebras understood as flow lines for global flows on   matrix Lie groups.

In the last section, besides briefly  summarizing the main guidelines along the work, we pose for further study continuous evolution algebras arising from dynamical systems obeying more general differential equations.

\section{Continuous evolution algebras}

Let $\mathbb{K}$ be $\mathbb{R}$ or $\mathbb{C}$. Matrices   will be denoted boldface.

 An $n$-dimensional evolution  $\mathbb{K}$-algebra
   ${\cal  E}$  is a  $\mathbb{K}$-vector space endowed with a multiplication, with respect to a
     natural basis ${\cal B}=\{e_1,\dots, e_n\}$,  given by:
$$  \left\{
      \begin{array}{ll}
        e_i^2=e_ie_i=\sum_{j=1}^n a_{ij}e_j, & \hbox{$i=1,\ldots,n$;} \\
        e_ie_j=e_je_i=0, & \hbox{$i\neq j$.}
      \end{array}
    \right.
$$
 Dynamics of  evolution algebras arise from  their {\sl evolution operator} $L_e$, defined as the one-sided  multiplication   by the evolution  element $e=\sum_{i=1}^ne_i$ given by ${\cal  B}$ \cite{Tian EA}.    Continuity and dynamics of evolution operators were considered in \cite{mellon-velasco}.

 If
${\cal  E}$ is a Markov (real) evolution algebra, then its evolution operator $L_e$ satisfies  Chapman-Kolmogorov like equations \cite[4.1.3]{Tian EA} and its (Markov) structure matrix $\mathbf{A}=(a_{ij})_{i,j=1}^n $ becomes the transition probability matrix of the underlying HDT Markov chain \cite[4.1.2]{Tian EA}.

Following \cite{paniello-Markov}, we denote ${\cal  E}(t)=\{{\cal  E}_t=({\cal  E},m(t))\mid t\in \mathbb{R}\}$ the family of $n$-dimensional evolution $\mathbb{K}$-algebras, defined on the $\mathbb{K}$-vector space ${\cal E}$, with multiplications, for a fixed natural basis ${\cal  B}$:
$$  \left\{
      \begin{array}{ll}
        e_i\cdot_t e_i=\sum_{j=1}^n a_{ij}(t)e_j, & \hbox{$i=1,\ldots,n$;} \\
        e_i\cdot_t e_j=e_j \cdot_te_i=0, & \hbox{$i\neq j$,}
      \end{array}
    \right.
$$
and by $\mathbf{A}(t)=(a_{ij}(t))_{i,j=1}^n\in M_n(\mathbb{K})$ the structure matrix of ${\cal E}_t$, for all $t\in \mathbb{R}$.
Once the natural basis ${\cal B}$ is fixed, we identify  ${\cal  E}(t)$ to the set of $\mathbb{R}$-indexed structure matrices $\{\mathbf{A}(t)\mid t\in \mathbb{R}\}$.

\begin{definition}\label{DEF continuous EA} {\rm \cite[6.2.4]{Tian EA}} A {\sl continuous evolution algebra} is a family ${\cal E}(t)$ of evolution algebras  such that the functions $a_{ij}(t)$ are differentiable, for all $i,j=1, \ldots,n$.
 \end{definition}

The set $M_n(\mathbb{K})$, of $n\times n$ matrices with entries in $\mathbb{K}$, with the operator norm
  is a Banach algebra. In fact, any two norms in $M_n(\mathbb{K})$ give rise to the same topology. The
   {\sl general linear group} $GL_n(\mathbb{K})$, the group of all invertible matrices in $M_n(\mathbb{K})$, and all its subgroups carry the subspace topology inherited from $M_n(\mathbb{K})$.

\begin{definition}\label{DEF differentiable curve} {\rm \cite[Definition 2.11]{baker}}
A {\sl differentiable curve} in $M_n(\mathbb{K})$ is a curve $\gamma:(a,b)\to M_n(\mathbb{K})$ whose derivative
$$ \gamma'(t)=\lim_{h\to 0}\frac{\gamma(t+h)-\gamma(t)}{h}$$  exists, in $M_n(\mathbb{K})$ for each $t\in (a,b)\subseteq \mathbb{R}$ (that is, componentwise, the maps $\gamma_{ij}(t)$ are differentiable, for all $i,j\in 1,\ldots, n$).
\end{definition}

Fixed a  natural basis ${\cal B}$, we will consider continuous evolution algebras to be defined by differentiable curves in $ M_n(\mathbb{K})$.

\begin{theorem}\label{TH C-EA + diff curves} ${\cal E}(t)$  is a continuous   evolution algebra
 if and only if the map  $\gamma:\mathbb{R}\to M_n(\mathbb{K})$, given by $\gamma(t)=\mathbf{A}(t)$ defines a differentiable curve.
\end{theorem}

\begin{proof} If suffices to recall that continuous  evolution algebras are  defined by   componentwise differentiable
 structure matrices   $\mathbf{A}(t)=(a_{ij}(t))_{i,j=1}^n\in M_n(\mathbb{K})$, defining   differentiable matrix-valued functions.
 \end{proof}

\begin{corollary}\label{COR time-invariant} Any (time-invariant) evolution algebra ${\cal E}$, with  structure matrix $\mathbf{A}$, defines a continuous evolution algebra ${\cal E}(t)$ such that $ {\cal E}_t= {\cal E}$ (equivalently $\mathbf{A}(t)=\mathbf{A}$) for all $t\in \mathbb{R}$.
\end{corollary}

\begin{proof} Consider the differentiable curve $\gamma(t)=\mathbf{A}$  for all $t\in \mathbb{R}$.
\end{proof}

\begin{example}
Let $\mathbf{A}\in M_n(\mathbb{K})$. A quite simple example of (non time-invariant) differentiable curve is $\gamma: \mathbb{R}\to  M_n(\mathbb{K})$, defined as $\gamma(t)=\mathbf{I}_n+t\mathbf{A}$, with $\gamma(0)=\mathbf{I}_n$ and  $\gamma'(0)=\mathbf{A}$.  It defines a continuous evolution algebra ${\cal E}(t)$ with multiplication $e_i\cdot_t e_j= \sum_{k=1}^n (\delta_{ik}+a_{ik}t)e_k$, if $i=j\in\{1,\ldots,n\}$, and $e_i\cdot_t e_j=0$ otherwise.
Here, non-singularity of $\gamma(t)$, i.e. of $\mathbf{A}(t)$, is only ensured when
$| t|<\left\{\frac{1}{ |\lambda|} \mid\hbox {$\lambda$ a nonzero eigenvalue of $\mathbf{A}$} \right\}$ \cite[p.~76]{baker}. Recall this implies that   algebras ${\cal E}_t$ are perfect.
\end{example}

The matrix series $\exp(\mathbf{X})=e^\mathbf{X}=\sum_{m=0}^\infty\frac{1}{m!}\mathbf{X}^m$ converges absolutely for any $\mathbf{X}\in M_n(\mathbb{K})$. It defines the {\sl matrix exponential function} $\exp: M_n(\mathbb{K})\to GL_n(\mathbb{K})$   \cite[Theorem~2.11]{hall}, which is infinitely differentiable (smooth), and injective on an open neighbourhood of $\mathbf{O}_n$,  the $n\times n$ matrix with all zero entries. Also it holds $\exp(\mathbf{X})^{-1}=\exp(-\mathbf{X})$ \cite[Theorem 2.3(3)]{hall}. Exponentials of matrices are easily computed for diagonalizable or nilpotent matrices.  SN decompositions of complex matrices \cite[Appendix B. Theorem B.6]{hall} and Jordan canonical forms \cite[Appendix B.4]{hall} are useful tools for matrix exponentiation \cite[Chapter 2, Sections 2.1. 2.2, 2.3]{baker}. A different procedure, based on solving an $n$th-order scalar differential equation derived from the matrix characteristic polynomial,  can be found in \cite[Appendix 5.14.1]{allen}.

\begin{example}\label{EX diff curves} \begin{enumerate}
\item[(i)] Let $\mathbf{X}\in  M_n(\mathbb{K})$. Then $\gamma (t)  =\exp(t\mathbf{X})$ defines a differentiable curve in $GL_n(\mathbb{K})$, with $\gamma(0)=\mathbf{I}_n$. Moreover, $\gamma'(t)=\gamma(t)\mathbf{X}= \mathbf{X}\gamma(t)$ with $\gamma'(0)=\left(\frac{d}{dt} \gamma(t)\right)_{\mid t=0}=\mathbf{X}$,
and $\gamma(s+t)=\gamma(s)\gamma(t)$ for all $s,t\in \mathbb{R}$ \cite[Proposition~2.3(4), Proposition 2.4]{hall}.
\item[(ii)]  Consider the flip flop process given by a Poisson process of intensity $\lambda>0$. Its standard stochastic semigroup is
$$\mathbf{A}(t)=\frac{1}{2} \left(
     \begin{array}{cc}
       1+e^{-2\lambda t} & 1-e^{-2\lambda t} \\
       1-e^{-2\lambda t} & 1+e^{-2\lambda t} \\
     \end{array}
   \right),\quad\hbox{for all $t\geq0$},
 $$
with infinitesimal generator (or rate matrix) $ \mathbf{Q}=  \left(
     \begin{array}{cc}
       -\lambda & \lambda \\
       \lambda & -\lambda \\
     \end{array}
   \right)
 $ so that $\mathbf{A}(t)=\exp(t\mathbf{Q}$) \cite[Example 2.2, Exercise 8.2.4]{bremaud}.  Considered, by default, only for nonnegative time values ($t\geq0$), (finite-state)  standard stochastic semigroups define  continuous time Markov processes.
Clearly nonnegativity of the matrices $\mathbf{A}(t)$ is only ensured when
 $t\geq0$,  but the map  $\gamma(t)=\mathbf{A}(t)$ still defines a differentiable curve in $GL_n(\mathbb{R})$ such that $\gamma(s+t)=\gamma(s)\gamma(t)$ for all $s,t\in \mathbb{R}$ \cite[Chapter 8, Section~2.2]{bremaud}.
Besides, as noted in the introduction, Forward and Backward Kolmogorov differential equations $\mathbf{A}'(t)=\mathbf{A}(t)\mathbf{Q}= \mathbf{Q}\mathbf{A}(t) $, with initial condition  $\mathbf{A}(0)=\mathbf{I}_n$, are satisfied.
 \end{enumerate}
 \end{example}

 \begin{lemma}\label{lemma example} Examples  \ref{EX diff curves}(i) and  \ref{EX diff curves}(ii) above define continuous evolution algebras.\end{lemma}

  We will say that a  continuous evolution algebra ${\cal  E}(t)$ {\sl satisfies a property~${\cal P}$} if each ${\cal  E}_t$ satisfies ${\cal P}$, for all $t\in \mathbb{R}$.  Thus
  ${\cal E}(t)$ is {\sl perfect} if ${\cal E}_t^2={\cal E}_t$  for all $t\in \mathbb{R}$, or, equivalently,  $\mathbf{A}(t)$ is nonsingular.  Continuous evolution algebras arising from the differentiable curves  in Example \ref{EX diff curves} are perfect.

The next result is straightforward.

 \begin{lemma}\label{lemma perfect continuous EA} Perfect continuous evolution algebras correspond to differentiable curves in $GL_n(\mathbb{K})$.  \end{lemma}

Given a matrix $\mathbf{X}\in M_n(\mathbb{K})$, the differentiable curve  $\gamma(t)=\exp(t\mathbf{X})$, defining  a continuous evolution algebra    ${\cal E}(t)$ with structure matrices $\mathbf{A}(t)= \gamma(t)=\exp(t\mathbf{X})$ (see, for instance, Example \ref{EX diff curves}) is the (unique) solution of the matrix ordinary differential equation (ODE):
$$\left\{
  \begin{array}{ll}
     \mathbf{A}'(t)=\mathbf{A}(t)\mathbf{X}, &  \\
    \mathbf{A}(0)=\mathbf{I}_n. &
  \end{array}
\right.$$

Approaching continuous evolution algebras as solutions of differential equations agrees to their original appearance in \cite{Tian EA} as tools to study the dynamics of particular genetic systems, and will be considered in the following sections.

\section{Continuous evolution algebras on matrix Lie groups}

In this section we consider    continuous evolution algebras arising as solutions of first order (matrix) ordinary differential equations.
The additional assumption on being  solutions of   ODEs  leads us to consider their structure matrices within  matrix Lie groups.

\begin{definition}\cite[Definition 1.4]{hall} A {\sl matrix Lie group} $G$ is a subgroup $G\leq GL_n(\mathbb{K})$, which is closed in $GL_n(\mathbb{K})$, that is, for any sequence $\{\mathbf{A}_m\}_{m\geq1}$ of matrices in $G$, converging, componentwise, as $m \to \infty$, to a matrix $\mathbf{A}$, then either $\mathbf{A}\in G$ or $\mathbf{A}$ is not invertible.
 \end{definition}

Besides the general linear groups $GL_n(\mathbb{K})$, examples of matrix Lie groups are the special linear groups $SL_n(\mathbb{K})=\{\mathbf{A}\in M_n(\mathbb{\mathbb{K}})\mid \det(\mathbf{A})=1\}$. If $\mathbb{K}=\mathbb{R}$,  the orthogonal group $O(n)=\{\mathbf{A}\in M_n(\mathbb{R})\mid \mathbf{A}^T\mathbf{A}=\mathbf{I}_n  \}$ and the special orthogonal group $SO(n)=O(n)\cap SL_n(\mathbb{R})$, and for $\mathbb{K}=\mathbb{C}$, the unitary group
$U(n)=\{\mathbf{A}\in M_n(\mathbb{C})\mid \mathbf{A}^\ast \mathbf{A}=\mathbf{I}_n  \}$ and the special unitary group $SU(n)=U(n)\cap SL_n(\mathbb{C})$  are also matrix Lie groups. The $n$-dimensional affine group
$$Aff_n(\mathbb{K})=\left\{
 \left(
   \begin{array}{cc}
     \mathbf{A} & \mathbf{t} \\
     0 & 1 \\
   \end{array}
 \right)\mid \mathbf{A}\in GL_n(\mathbb{K}), \mathbf{t}\in \mathbb{K}^n\right\}\leq GL_{n+1}(\mathbb{K});$$
consisting of the set of transformations in $\mathbb{K}^n$ of the form $\mathbf{x}\mapsto \mathbf{A}\mathbf{x}+\mathbf{t}$, with nonsingular $\mathbf{A}$, is a closed subgroup of $GL_{n+1}(\mathbb{K})$ and a matrix Lie group.

\begin{remark}\label{REMARK matrix Lie groups}
\begin{enumerate} \item[(i)] The matrix Lie groups $O(n)$, $SO(n)$, $U(n)$ and $SU(n)$ are compact, but $GL_n(\mathbb{K})$ and $SL_n(\mathbb{K})$  are not \cite[Definition 1.6, Examples 1.3.1 and 1.3.2]{hall}.
 \item[(ii)]  $GL_n(\mathbb{C})$ is connected, as well as $SO(n)$, $U(n)$, $SU(n)$ and $SL_n(\mathbb{K})$. $GL_n(\mathbb{R})$ and $O(n)$ have two connected components given by the sign of the determinant. The set of nonsingular real matrices with positive determinant $GL_n(\mathbb{R})^+$
is the connected component of $GL_n(\mathbb{R})$  containing the identity $\mathbf{I}_n$  \cite[Definition 1.7, Proposition 1.8 - Proposition~1.12]{hall}. Note $O(n)^+=O(n)\cap GL_n(\mathbb{R})^+=SO(n)$.
\item[(iii)] Matrix Lie groups are Lie groups, smooth manifolds \cite[Appendix~C.2.6]{hall}. Conversely, many, but not all,  Lie groups have a matrix form, \cite[Appendix C.3]{hall}. See also \cite[Section 7.7]{baker}.
\end{enumerate}
\end{remark}

\begin{definition}\label{DEF 1parameter subgroup} {\rm \cite[Definition 2.12]{hall}} Let $G$ be a matrix Lie group. A {\sl one-parameter subgroup} of $G$ is a map $\gamma:(\mathbb{R},+)\to G$ such that:
\begin{enumerate} \item[(i)]  $\gamma$ is continuous for all $t\in \mathbb{R}$.
\item[(ii)]  $\gamma(0)=\mathbf{I}_n$.
\item[(iii)] $\gamma(s+t)=\gamma(s)\gamma(t)$ for all $s,t\in \mathbb{R}$.
\end{enumerate}
\end{definition}

Some references consider the slightly different notion of {\sl one-parameter semigroup}  $\gamma(t)$  defined on intervals $(-\epsilon,\epsilon)\subseteq \mathbb{R}$, for some $\epsilon>0$.  Then   there exists a unique extension $\tilde{\gamma}:\mathbb{R}\to G$ to a one-parameter subgroup on $\mathbb{R}$ such that $\tilde{\gamma}(t)=\gamma(t)$ for all $t\in (-\epsilon,\epsilon)$
 \cite[Proposition 2.16]{baker}. We thus   consider the case $\epsilon=\infty$.

The next result follows from the characterization of differentiable one-parameter subgroups as smooth curves given by the exponential function \cite[Theorem 3.2.6]{hilgert-Neeb}.

\begin{theorem}\label{TH continuous EAs - 1p subgroups}   Any one-parameter subgroup $\gamma:(\mathbb{R},+)\to G$ in a matrix Lie group $G$ defines a perfect continuous evolution algebra  ${\cal E}(t)$, whose structure matrices $\mathbf{A} (t)$  are the unique solution of a matrix ODE:
$$
\left\{
  \begin{array}{ll}
   \mathbf{A}' (t)=\mathbf{A}(t)\mathbf{X}, &  \\
   \mathbf{A}(0)=\mathbf{I}_n,&
  \end{array}
\right.
$$ for some $\mathbf{X}\in M_n(\mathbb{K})$. Conversely, for any such perfect continuous evolution algebra  $\gamma(t)=\mathbf{A}(t) (=\exp(t\mathbf{X}))$  defines a one-parameter subgroup in $GL_n(\mathbb{K})$.
\end{theorem}

\begin{proof} Let $G\leq   GL_n(\mathbb{K})$  be a matrix Lie group, and $\gamma:(\mathbb{R},+)\to G$ be a one-parameter subgroup in $G$. By \cite[Theorem 3.2.6]{hilgert-Neeb} $\gamma(t)=\exp(t\mathbf{X})$, for some $\mathbf{X}\in M_n(\mathbb{K})$. By Theorem \ref{TH C-EA + diff curves} and Lemma \ref{lemma perfect continuous EA}, ${\cal E}(t)$ with structure matrices $\mathbf{A}(t)=\gamma(t) $ is a perfect continuous evolution algebra. The fact that matrices $\mathbf{A}(t)$ are the unique solution of the matrix ODE is clear. For the converse statement see again \cite[Theorem 3.2.6]{hilgert-Neeb}.
\end{proof}

  We will say that a continuous evolution algebra is {\sl $G$-continuous} if all its structure matrices $\mathbf{A}(t)\in G$, for all $t\in \mathbb{R}$, and {\sl one-parameter $G$-continuous} if they are as in Theorem \ref{TH continuous EAs - 1p subgroups}. Clearly one can always consider  $G \leq GL_n(\mathbb{C})$. $G$-continuous evolution algebras are perfect.

   \begin{example}\label{ex SO(2)}  The one-parameter subgroup $\gamma:(\mathbb{R},+)\to GL_2(\mathbb{R})$ given by
   $\gamma(t)=\left(
                \begin{array}{cc}
                \cos t &  \sin t \\
                     -\sin t & \cos t
                \end{array}
              \right)$
     defines a  2-dimensional $SO(2)$-continuous evolution algebra ($im \  \gamma=SO(2)$). For a fixed natural basis ${\cal B}=\{e_1, e_2\}$, the multiplication in
 ${\cal E}(t)$ is given by:
$$  \left\{
      \begin{array}{ll}
        e_1\cdot_t e_1= e_1 \cos t +e_2\sin t, &  \\
       e_2\cdot_t e_2= -e_1\sin t +e_2\cos t.&
      \end{array}
    \right.
$$
Here $\gamma(t)=\exp(t\mathbf{X} )$, where $ \mathbf{X}=\gamma'(0)$ is the skew-symmetric matrix $\left(\begin{array}{cc}
0 &  1 \\
  -1 & 0
  \end{array}
  \right)$.
 \end{example}

The (real) Lie algebra   of a matrix Lie group $G$ is   $\mathfrak{g}=\{\mathbf{X}\in M_n(\mathbb{K})\mid \exp(\mathbb{R}\mathbf{X})\subseteq G\}$ \cite[Subsection 2.15]{hall}. This is the {\sl tangent space} to $G$ at $\mathbf{I}_n$:
$$ T_{ \mathbf{I}_n} G=\{\gamma'(0)\in M_n(\mathbb{K})\mid \hbox{$\gamma$ is a differentiable curve in $G$ with $\gamma(0)=\mathbf{I}_n$}\}.$$
Derivatives at origin of differentiable curves on matrix Lie groups with $\gamma(0)=\mathbf{I}_n$ (i.e. elements of the Lie algebra $\mathfrak{g}$) are called {\sl velocity vectors}.

  \begin{theorem}\label{TH bijection Lie algebra - Gcontinuous EA}  Let $G$ be a matrix Lie group.  Each  $\mathbf{X}\in \mathfrak{g}$ defines a
(one-parameter) $G$-continuous evolution algebra
 with velocity $\mathbf{X}$ at origin.
  \end{theorem}

  \begin{proof}  Let $\mathfrak{g}$ be the Lie algebra of $G$, and let us denote by $Hom(\mathbb{R},G)$   the set of all continuous group homomorphisms $(\mathbb{R},+)\to G$. By \cite[Lemma~4.1.5]{hilgert-Neeb}, there exists a bijection $\mathfrak{g}\to Hom(\mathbb{R},G)$, given by  $\mathbf{X}\mapsto \gamma_\mathbf{X}$ with $\gamma_\mathbf{X}(t)=\exp(t\mathbf{X})$. The result now follows from Theorem \ref{TH continuous EAs - 1p subgroups}.
      \end{proof}

  \begin{corollary}\label{COR Lie algebra ODE}  Let $\mathbf{X}\in \mathfrak{g}$. The   $G$-continuous evolution algebra ${\cal E}(t)$ with structure matrices   $\mathbf{A}(t)=\exp(t\mathbf{X})$  is the unique solution of the ODE problem:
$$
\left\{
  \begin{array}{ll}
   \mathbf{A}' (t)=\mathbf{A} (t)\mathbf{X}, &  \\
    \mathbf{A} (0)=\mathbf{I}_n. &
  \end{array}
\right.
$$
      \end{corollary}

\begin{proof} See Theorem \ref{TH continuous EAs - 1p subgroups} and note
  that now   $ \mathbf{X}\in \mathfrak{g}$ ensures   ${\cal E}(t)$ is $G$-continuous.
\end{proof}

 Let  ${\cal E}(t) $ be as in Corollary \ref{COR Lie algebra ODE}. Then its structure matrices are not only nonsingular matrices   in $G\leq GL_n(\mathbb{K})$, but all them fall inside   the same connected component of $G$, that containing $ \mathbf{A}(0)=\mathbf{I}_n$. Thus,  if $\mathbb{K}=\mathbb{R}$, we have
      $ \mathbf{A}(t) \in G\cap GL_n(\mathbb{R})^+$   (matrices with positive determinant). See, for instance, Example \ref{ex SO(2)}, where
      $ im\  \gamma \subseteq O(2)\cap GL_2(\mathbb{R})^+=SO(2)$.   A similar situation occurs when considering the Lorenz group $Lor(1,1)$ inside the generalized orthogonal group $O(1,1)$. Only one of the four connected components of $O(1,1)$ is reached.

\begin{example}\label{EX Lor(1,1)}
Let $G=Lor(1,1)=  \left\{\left(\begin{array}{cc}
 \cosh t &   \sinh t \\
    \sinh t &  \cosh t
  \end{array}
  \right) \mid    t\in \mathbb{R}\right\}$, and consider   the one-parameter subgroup $\gamma(t)=\exp(t\mathbf{X})$, with velocity vector at origin $ \mathbf{X}= \left(\begin{array}{cc}
0 &  1 \\
   1 & 0
  \end{array}
  \right)$. It defines a 2-dimensional real $G$-continuous evolution algebra with structure matrices
 $ \mathbf{A}(t)=\exp(t\mathbf{X})= \left(\begin{array}{cc}
 \cosh t &   \sinh t \\
    \sinh t &  \cosh t
  \end{array}
  \right)$, for all $t\in \mathbb{R}$. Thus,  ${\cal E}_t$ has multiplication:
$$  \left\{
      \begin{array}{ll}
        e_1\cdot_t e_1= e_1 \cosh t +e_2\sinh t, &  \\
       e_2\cdot_t e_2=  e_1\sinh t +e_2\cosh t.&
      \end{array}
    \right.
$$
The Lorentz group $Lor(1,1)$ is connected, with Lie algebra
$
\mathfrak{lor}(1,1)=\left\{ t \mathbf{X} \mid    t\in \mathbb{R}\right\}$, and surjective exponential map $\exp:\mathfrak{lor}(1,1)\to Lor(1,1)$ \cite[Remark~6.5, Theorem~6.7, Theorem~6.11]{baker}. However considered within $O(1,1)$ (see Example \ref{EX O(1,1)}),  $ \mathbf{A}(t)\in Lor(1,1)\subsetneqq O(1,1)^+=SO(1,1)=O(1,1)\cap GL_2(\mathbb{R})^+$.
\end{example}

 \begin{example}\label{EX CSV} The classification of (time-invariant) perfect three dimensional (real or complex) evolution algebras follows from \cite[Theorem 3.5(iv)]{cabrera-siles-velasco2017}. We point out here that in \cite{cabrera-siles-velasco2017} multiplication constants are arranged into the structure matrices by columns, whereas here, following \cite[Definition 3, p.~20]{Tian EA},  we are considering them row-rearranged. Families of non-isomorphic  3-dimensional perfect real evolution algebras are detailed in \cite{cabrera-siles-velasco2017 arxiv}. Let $\mathbb{K}=\mathbb{R}$, and consider  the second family (second row) in \cite[Table 18]{cabrera-siles-velasco2017 arxiv}, whose structure matrices have negative determinant.  Although any of this algebras trivially defines a continuous evolution algebra  (see Corollary \ref{COR time-invariant}), such real evolution algebras cannot appear within (for any $t\in \mathbb{R}$)  algebras as in  Corollary \ref{COR Lie algebra ODE}. The same applies to 2-dimensional real evolution algebras of type $E_7(a_4)$ \cite{murodov}. For any differentiable map $\alpha:\mathbb{R}\to \mathbb{R}$, the matrix-valued function
 $ \mathbf{A}(t)= \left(\begin{array}{cc}
0 &  1 \\
   1 & \alpha(t)
  \end{array}
  \right)$  defines a perfect $GL_2(\mathbb{R})$-continuous evolution algebra, whose
  structure matrices have negative determinant.
\end{example}

To overcome  drawbacks  as in Examples \ref{EX Lor(1,1)} and \ref{EX CSV}, let us next consider  matrix ODEs with  initial condition different from $\mathbf{A}(0)=\mathbf{I}_n$. Given arbitrary matrices $\mathbf{X}$ and $\mathbf{A}_0$ in $M_n(\mathbb{K})$, the matrix ODE:
$$
\left\{
  \begin{array}{ll}
   \mathbf{A}' (t)=\mathbf{A}(t)\mathbf{X}, &  \\
    \mathbf{A} (0)=\mathbf{A}_0, &
  \end{array}
\right.
$$
  defines a differentiable curve $\gamma(t)= \mathbf{A}_0\exp(t\mathbf{X})$ in $M_n(\mathbb{K})$ and, therefore a  continuous evolution algebra ${\cal E}(t)$. Only  if $\mathbf{A}_0\in GL_n(\mathbb{K})$, then $\mathbf{A}(t)= \mathbf{A}_0\exp(t\mathbf{X}) \in GL_n(\mathbb{K})$ for all $t\in \mathbb{R}$, being then ${\cal E}(t)$ perfect.

\begin{proposition}\label{PROP perfect}  Let ${\cal E}(t)$ be a  continuous evolution algebra arising as the solution of the ODE
$$
\left\{
  \begin{array}{ll}
   \mathbf{A}' (t)= \mathbf{A}(t)\mathbf{X}, &  \\
    \mathbf{A} (0)=\mathbf{A}_0. &
  \end{array}
\right.
$$
Then:
\begin{enumerate}
\item[(i)] ${\cal E}(t)$  is perfect if and only if  $\mathbf{A}_0\in GL_n(\mathbb{K})$. If moreover $\mathbb{K}=\mathbb{R}$, then
 for all $t\in \mathbb{R}$, it holds $\mathbf{A}(t)\in GL_n(\mathbb{R})^\sigma$, where $GL_n(\mathbb{R})^\sigma=\det^{-1}\mathbb{R}^\sigma$, and $\sigma=\pm$ is the sign of $\det(\mathbf{A}_0)$.
 \item[(ii)] Otherwise, if $\mathbf{A}_0\not\in GL_n(\mathbb{K})$,   ${\cal E}_t$ is not perfect, for any $t\in \mathbb{R}$.
\end{enumerate}
\end{proposition}

\begin{proof} If follows from  $ \mathbf{A}(t)=\mathbf{A}_0\exp(t\mathbf{X})$, and
the continuity of the determinant map in  $ GL_n(\mathbb{R})$.
\end{proof}

\begin{example}\label{EX O(1,1)} Elements of the generalized orthogonal group $O(1,1)$ are of the form \cite[Example 6.4]{baker}:
\begin{align*}
& \mathbf{A}(t)=\exp(t\mathbf{X}),\quad \hbox{with $\mathbf{X}$ as in  Example \ref{EX Lor(1,1)}},\\
&\mathbf{A}_2(t)=   \left(
                      \begin{array}{cc}
                        -1 & 0 \\
                        0 & -1 \\
                      \end{array}
                    \right)
\mathbf{A}(t)=\mathbf{A}_2\mathbf{A}(t),\\
&\mathbf{A}_3(t)=   \left(
                      \begin{array}{cc}
                         1 & 0 \\
                        0 &  -1 \\
                      \end{array}
                    \right)
\mathbf{A}(t)=\mathbf{A}_3\mathbf{A}(t),\\
&\mathbf{A}_4(t)=   \left(
                      \begin{array}{cc}
                        -1 & 0 \\
                        0 &  1 \\
                      \end{array}
                    \right)
\mathbf{A}(t)=\mathbf{A}_4\mathbf{A}(t),\end{align*}
For $i=1,2,3,4$ ($\mathbf{A}_1=\mathbf{I}_n$),   $\mathbf{A}_i(t)=\mathbf{A}_i\exp(t\mathbf{X})$ defines  a $O(1,1)$-continuous evolution algebra, as in  Proposition \ref{PROP perfect},  contained in each one of the four different connected components of  $O(1,1)$.
\end{example}

\begin{example} Consider the matrices $ \mathbf{A}_0= \left(\begin{array}{cc}
0 &  1 \\
   1 & 0
  \end{array}
  \right)$  and $ \mathbf{X}= \left(\begin{array}{cc}
0 &  \alpha \\
   0 & 0
  \end{array}
  \right)$. Then
 $ \mathbf{A}_1(t)=  \exp(t  \mathbf{X}) =\left(\begin{array}{cc}
1 &  t\alpha \\
   0 & 1
  \end{array}
  \right)$ defines a $SL_2(\mathbb{R})$-continuous evolution algebra, whereas
$ \mathbf{A}_2(t)=  \mathbf{A}_0\exp(t  \mathbf{X})$ gives rise to a continuous evolution algebra, with ${\cal E}_t$ of type $E_7(t\alpha)$ (see Example \ref{EX CSV}).
\end{example}

The next Corollary is clear.

\begin{corollary}\label{COR tangent bundle} If $\mathbf{A}_0\in G$ and $\mathbf{X}\in \mathfrak{g}$,   the continuous evolution algebra  ${\cal E}(t)$, with $ \mathbf{A}(t)=\mathbf{A}_0\exp(t\mathbf{X})$ is $G$-continuous and lies in the connected component of $\mathbf{A}_0$.
\end{corollary}

One of the main consequences of this section  is formulated in the following Remark~\ref{RM tangent bundle} for further reference.

\begin{remark}\label{RM tangent bundle} Corollary \ref{COR tangent bundle} tells us that any element of the tangent bundle $T(G)$ of the matrix Lie group $G$ defines a $G$-continuous evolution algebra. Indeed, recall $T(G)$ is the disjoint union of the tangent spaces $T_{\mathbf{B}}G$ at $\mathbf{B}\in G$, that is:
 $T(G)=\bigsqcup T_{\mathbf{B}}G=\left\{(\mathbf{B},\mathbf{V})\mid \mathbf{B}\in G, \ \mathbf{V}\in T_{\mathbf{B}}G) \right\}$,
where $T_{\mathbf{B} }G=\mathbf{B} \mathfrak{g} $. Given $(\mathbf{B},\mathbf{V})\in T(G)$, it suffices to consider ${\cal E}(t)$ with structure matrices $\mathbf{A}(t)=\mathbf{B}\exp(t\mathbf{X})$, and $\mathbf{X}=\mathbf{B}^{-1}\mathbf{V}$.
\end{remark}

 To conclude this section we provide two examples   deeply relying on the properties of $G$.
In Example \ref{EX H(3)} the notion of $G$-continuous is relaxed to consider continuous evolution algebras on (non-necessarily matrix) Lie groups.

  \begin{example}\label{EX H(3)}
  The Heisenberg group $H(3)$ is the  (non-matrix \cite[Theorem~7.36]{baker}) Lie group:
  $$ H(3)=\left\{
    \mathbf{A}(\alpha,\beta,\delta)= \left(
  \begin{array}{ccc}
    1 & \alpha &\beta \\
    0 & 1 & \delta \\
    0 & 0 & 1
  \end{array}\right)
  \mid
  \alpha, \beta,\delta\in \mathbb{R}
    \right\} $$
with Lie algebra $\mathfrak{g}=\left\{
   \mathbf{X} (a,b,c)= \left(
  \begin{array}{ccc}
    0 & a &b \\
    0 & 0 & c \\
    0 & 0 & 0
  \end{array}\right)
  \mid
  x, y, z\in \mathbb{R}
    \right\} $   \cite[p.~162]{hall}. Here the exponential map  $\exp: \mathfrak{g}\to H(3)$  is  one-to-one and onto \cite[Exercise~II.26]{hall},
    and  $\exp(   \mathbf{X} (a,b,c) )=  \mathbf{A}(a,b+\frac{1}{2}ac,c)$   \cite[Vol.~2, 10.6.1]{Chirikjian}, providing an explicit description for  $H(3)$-continuous evolution algebras similar to those appearing in Theorem~\ref{TH bijection Lie algebra - Gcontinuous EA}.
    More generally, given
        differentiable maps  $ \alpha, \beta, \delta:\mathbb{R}\to\mathbb{ R}$, then
    $ \mathbf{A}(t)= \mathbf{A}(\alpha(t),\beta(t),\delta(t))$ defines  a $H(3)$-continuous evolution algebra, with multiplication (w.r.t a natural basis ${\cal B}=\{e_1,e_2,e_3\}$:
$$
\left\{
  \begin{array}{rrrr}
    e_1\cdot_t e_1= & e_1  &  +\alpha(t)e_2  & +\beta(t)e_3, \\
    e_2\cdot_t e_2= &     &    e_2           &  +\delta(t)e_3,  \\
    e_3\cdot_t e_3= &     &               &  e_3 \\
  \end{array}
\right.
$$
Conversely,
     $$  \mathbf{A}(t)=\exp\left(
  \begin{array}{ccc}
    0 & a(t) & b(t) \\
    0 & 0 &c(t) \\
    0 & 0 & 0
  \end{array}\right)= \mathbf{A}\left(a(t), b(t)+\frac{1}{2}a(t)c(t),c(t)\right),$$
 also defines a $H(3)$-continuous evolution algebra  for any  differentiable maps $ a, b, c:\mathbb{R}\to\mathbb{R}$.
 \end{example}

The exponential map is not however surjective, nor injective, on any matrix Lie group $G$. If $G$ is compact and connected, then any element of $G$ does fall within some one-parameter subgroup and can therefore be embedded (mirroring Markov chains terminology \cite{paniello-Markov}) into a  $G$-continuous evolution algebra as those in Corollary \ref{COR Lie algebra ODE}, but otherwise it would depend on $G$.
Matrices in the previous Example \ref{EX H(3)}  have also a form
  $  \mathbf{A}(t)= \mathbf{A}(\alpha(t),\beta(t),\delta(t))=\exp(\beta(t)E_{13}) \exp(\delta(t)E_{23}) \exp(\alpha(t)E_{12})$,   as product  of exponential matrices. (Here $E_{ij}$ stands for the usual matrix units.)
 In fact, there is a number of matrix Lie groups whose elements admit similar parametrizations.

 \begin{example}\label{EX SL2}
 Consider the special linear group
  $ SL_2(\mathbb{R})$, whose Lie algebra $\mathfrak{sl}_2(\mathbb{R})$ is the set of traceless $2\times 2$ real matrices. The exponential map $ \exp :\mathfrak{sl}_2(\mathbb{R})\to SL_2(\mathbb{R})$ is not surjective. Indeed, for any $\mathbf{X}\in \mathfrak{sl}_2(\mathbb{R})$, $tr(e^\mathbf{X}) \geq -2$, and it is not difficult to find $\mathbf{A}\in SL_2(\mathbb{R})$ with $tr (\mathbf{A})< -2$ \cite[Problem II.2.2]{gallier}.  Since   $\mathfrak{sl}_2(\mathbb{R})$ is semisimple \cite[p.~162]{hall}, using Iwasawa decomposition \cite[Theorem~6.46]{knapp},  any
   $\mathbf{A}\in SL_2(\mathbb{R})$ decomposes as a product of exponential matrices $\mathbf{A}=\exp(\alpha E_1) \exp(\beta E_2)\exp(\delta E_3)$, where
$\{E_1=E_{21}-E_{12}, E_2=E_{11}-E_{22}, E_3=E_{12}\}$ is a basis of $\mathfrak{sl}_2(\mathbb{R})$  \cite[Vol 2, Subsection 10.6.3]{Chirikjian}.
 Thus, for any differentiable maps $ \alpha, \beta, \delta:\mathbb{R}\to\mathbb{ R}$,
\begin{align*} \mathbf{A}(t)&= \mathbf{A}(\alpha(t),\beta(t),\delta(t))  =\exp(\alpha(t) E_1) \exp(\beta(t) E_2)\exp(\delta(t) E_3)=\\
&=\left(\begin{array}{cc}
 \cos \alpha(t) &   - \sin  \alpha(t) \\
    \sin  \alpha(t) &  \cos \alpha(t)
  \end{array}
  \right)
\left(\begin{array}{cc}
 e^{\beta(t)}  &   0 \\
   0 &  e^{-  \beta(t) }
  \end{array}
  \right)\left(\begin{array}{cc}
 1 &   \delta(t) \\
    0 &  1
  \end{array}
  \right)
\end{align*}
 defines a $SL_2(\mathbb{R})$-continuous evolution algebra.
 \end{example}

\section{Continuous time Markov evolution algebras}

Continuous time Markov (CT-Markov) evolution algebras yield on standard stochastic semigroups \cite[Definition 6.1(i)]{paniello-Markov}.
In this section we settle CT-Markov EAs as continuous evolution algebras. Along this section $\mathbb{K}=\mathbb{R}$.

\begin{theorem}\label{TH CT-Markov EA} Continuous time Markov evolution algebras are perfect real continuous evolution algebras.
\end{theorem}

\begin{proof}
Let ${\cal E}(t)$  be a CT-Markov evolution algebra with defining standard stochastic semigroup $\{\mathbf{A}(t)\}_{t\geq0}$. Since ${\cal E}(t)$  is finite-dimensional,   $\{\mathbf{A}(t)\}_{t\geq0}$ satisfies both Forward and Backward Kolmogorov differential equations \cite[Subsection 3.3]{bremaud}. Thus $ \mathbf{A}(t)=\exp(t\mathbf{Q})$, where $\mathbf{Q}$ denotes the process rate matrix,
  defines a, clearly perfect,  continuous evolution algebra  \cite[Theorem~2.21]{bremaud}.
\end{proof}

\begin{corollary} Continuous time Markov evolution algebras define one-parameter subgroups.
\end{corollary}

\begin{proof}
Let  $\gamma(t)=\mathbf{A}(t)=\exp(t\mathbf{Q})$ be as in the proof of Theorem \ref{TH CT-Markov EA}. It defines a differentiable, thus continuous, curve, with $\gamma(0)=\mathbf{I}_n$.  Moreover, see \cite[Proposition 2.1(i)]{baker}, $\exp((t+s)\mathbf{Q})=\exp(t\mathbf{Q}) \exp(s\mathbf{Q})$, for all $t,s\in \mathbb{R}$, implying  the so-called group homomorphism property   $\gamma(t+s)=\gamma(t)\gamma(s)$. Hence, see Definition \ref{DEF 1parameter subgroup}, $\gamma(t)$ defines a one-paramater subgroup in $GL_n(\mathbb{R})$.
 \end{proof}

\begin{remark}\begin{enumerate}\item[(i)]
Under the current approach, rate matrices, i.e. infinitesimal generators, $\mathbf{Q}$ of CT-Markov evolution algebras become tangent or velocity vectors
  at the identity $\mathbf{I}_n$ (for $t=0$).
\item[(ii)] Given a nonsingular Markov matrix, that is, a nonnegative matrix in the stochastic group $S(n,\mathbb{R})$  of nonsingular real matrices with row sums equal to one, its inverse matrix does not need to be nonnegative (i.e. still a Markov matrix).  Standard stochastic semigroups only ensure  the nonnegativity of matrices $\gamma(t)=\exp(t\mathbf{Q})$ for nonnegative time values $t\geq0$, whereas
  $\gamma(-t)=\exp(-t\mathbf{Q})=\exp( t\mathbf{Q})^{-1}=\gamma(t)^{-1}\in S(n,\mathbb{R})$  may not be nonnegative   \cite[Proposition 2.3(3)]{hall}. See, for instance, Example~\ref{EX diff curves}(ii).
\item[(iii)] Any $n$-dimensional CT-Markov evolution algebra defines a one-parameter subgroup  on the $(n-1)$-dimensional affine group, and, therefore, an $Aff_{n-1}(\mathbb{R})$-continuous evolution algebra. Recall   $S(n,\mathbb{R}) \cong Aff_{n-1}(\mathbb{R})$ \cite[Theorem~B]{poole}.
 \end{enumerate}
\end{remark}

Of special interest are   continuous evolution algebras defined on    generalized doubly-stochastic Lie groups \cite[Proposition~3.1, Proposition 3.2]{mourad}.

\begin{example} Let $\mathbf{A}(t)$ and $\mathbf{Q}$ be as in Example \ref{EX diff curves}(ii). Then $\mathbf{A}(t)$ is a nonsingular 1-generalized doubly stochastic (i.e.
 real with row and column sums equal to 1) matrix in the Lie group $G=\Omega^1(2,\mathbb{R})\cap GL_2(\mathbb{R})$, whereas the rate matrix $\mathbf{Q}$  is in its Lie algebra $\mathfrak{g}=\Omega^0(2,\mathbb{R})$ \cite[Proposition~3.1]{mourad}.
\end{example}

Given a  CT-Markov evolution algebra, the extension to negative times, from the stochastic process to the differentiable curve,   considered in Theorem \ref{TH CT-Markov EA}, is related to the notion of reversibility of   continuous-time Markov chains \cite[Chapter 8, Section 5.2]{bremaud}.
See, for instance, the Reversibility Theorem for irreducible and ergodic regular jump HMC on a finite state space \cite[Chapter 8, Theorem 5.4]{bremaud} which provides sufficient conditions on the infinitesimal generator  $\mathbf{Q}$ (the so-called detailed-balance equations) for the reverse process to be right-continuous and statistically equivalent to the direct process (see also \cite[Theorem 7.1]{pardoux}).

\begin{example} Reverse-time continuous Markov processes define CT-Markov evolution algebras and therefore continuous evolution algebras. This includes   examples of birth and death process.  Even if many reversible processes are defined on infinite state spaces, under   additional assumptions, truncations of reversible processes remain reversible, providing examples of (finite dimensional) continuous evolution algebras arising from many different areas of research.
\end{example}

\section{Evolution algebras and vector fields}

Evolution algebras were first considered to model    dynamics in non-mendelian genetics. Since system  evolution is
often described by differential equations, and     working with evolution algebras usually comes with
 dependence on in advance fixed natural bases, matrix Lie groups   seemed  to be a natural framework where considering continuous evolution algebras. Even if  this restricts ourselves to   perfect algebras.

From the (matrix) Lie group viewpoint, $G$-continuous evolution algebras as in Corollary \ref{COR Lie algebra ODE} are   (maximal) integral curves for left-invariant (thus complete) vector fields on $G$. In fact, for any $\mathbf{X}\in \mathfrak{g}$, $\gamma_\mathbf{X}(t)=\exp_G(t\mathbf{X})$ is the unique solution of the problem:
 $$\left\{
   \begin{array}{ll}
     \gamma'(t)={\cal X} (\gamma(t)) \  (=\gamma(t)\mathbf{X}), &   \\
     \gamma(0)=\mathbf{I}_n, &
   \end{array}
 \right.
$$
where ${\cal X} :G\to T(G)$ is the   left-invariant vector field on $G$ with ${\cal X} (\mathbf{I}_n)=\mathbf{X}\in T_{\mathbf{I}_n}G$. Recall that   left-invariant vector fields  are determined by their value at the group identity element, which results in identifying the set of all left-invariant vector fields with  $T_{\mathbf{I}_n}G = \mathfrak{g}$. This leads us to approach   previous results in terms of flow lines.

\begin{definition}\cite[Definition 8.5.9]{hilgert-Neeb}  Let $G$ be a matrix Lie group. A {\sl global flow} on $G$  is   a smooth map $\Phi:\mathbb{R}\times G\to G$, such that $\Phi(0,\mathbf{A} )=\mathbf{A}$ and $\Phi(s,\Phi(t,\mathbf{A} ))=\Phi ( s+t ,\mathbf{A})$ for all $s,t\in \mathbb{R}$ and $\mathbf{A}\in G$.  Maps $\Phi_\mathbf{A}:\mathbb{R} \to G$  defined by $ \Phi_\mathbf{A}(t)=\Phi(t,\mathbf{A})$ are called  {\sl flow lines}.\end{definition}

\begin{remark}\label{RM left-invariante VF- flow}
  Let ${\cal X}  $ be the left-invariant vector field on $G$ with ${\cal X} (\mathbf{I}_n)=\mathbf{X}\in \mathfrak{g}$.  It defines a global flow
$\Phi^{{\cal X}}:\mathbb{R}\times G\to G$, with  $\Phi^{{\cal X}}(t, \mathbf{A})=\mathbf{A}\gamma_\mathbf{X}(t)=\mathbf{A}\exp_G(t\mathbf{X})$. Then for each $\mathbf{A}\in G$, the flow line   $(\Phi^{{\cal X}})_\mathbf{A}(t)=\Phi^{{\cal X}}(t,\mathbf{A})$ is an  integral curve  for  ${\cal X}$, and a perfect $G$-continuous evolution algebra.
  See \cite[Lemma 9.2.4]{hilgert-Neeb}.
\end{remark}

Not all vector fields (resp. flows) are complete (resp. global). This holds, for instance, when $G$ is compact \cite[Corollary 8.5.8]{hilgert-Neeb}, but   existence of integral curves is, in general, only ensured locally.
 Thinking of vector fields as first order differential operators, this amounts to have a first order ordinary differential equation having no solution defined for all $t\in \mathbb{R}$.
In what follows we will restrict to   complete vector fields, with associated global flows.

 Let now
 $\Phi: \mathbb{R}\times G\to G$ be a global flow. It
carries a complete (smooth) vector field ${\cal X}^\Phi(\mathbf{A})=\left(\frac{d}{dt}\right)_{\mid t=0} \Phi(t,\mathbf{A})
=\Phi_\mathbf{A}'(0)\in T_\mathbf{A}G$, whose maximal integral curves equal the flow lines $
\Phi_\mathbf{A}(t)=\Phi(t,\mathbf{A})$. Then
  $\Phi_\mathbf{A}(0)=\Phi(0,\mathbf{A})=\mathbf{A}$ and
$\Phi'_\mathbf{A}(0)={\cal X}^\Phi(\mathbf{A})$  for all $\mathbf{A}\in G$ \cite[Lemma 8.5.10]{hilgert-Neeb}.

\begin{theorem}\label{TH flows}
Let $\Phi$ be a global flow on a matrix Lie group $G$ with associated complete vector field  ${\cal X}^\Phi$. Then for any $\mathbf{A}\in G$, $\Phi_\mathbf{A}(t)=\Phi(t,
\mathbf{A})$ defines a perfect $G$-continuous evolution algebra with
$\Phi_\mathbf{A}(0)= \mathbf{A}$ and $\Phi'_\mathbf{A}(0)={\cal X}^\Phi(\mathbf{A})$.
\end{theorem}

\begin{proof} It suffices to notice that
 flows are uniquely determined by their vector fields \cite[Theorem 8.5.12]{hilgert-Neeb}.
 \end{proof}

Theorem \ref{TH flows} above provides a source of continuous evolution algebras arising as smooth solutions for differential operators  given as vector fields on the matrix Lie group $G$.

\begin{corollary}\label{COR final} Let ${\cal E}$ be a $n$-dimensional $\mathbb{K}$-vector space with basis ${\cal B}$, and let
$\Phi$ and ${\cal X}^\Phi$ be as in Theorem \ref{TH flows}. Then ${\cal E}(t)$ with structure matrices (w.r.t. ${\cal B}$) $
\mathbf{A}(t)=\Phi_\mathbf{A}(t)$ is a perfect $G$ continuous evolution algebra  such that:
\begin{enumerate}
\item[(i)] $\mathbf{A}(0)=\mathbf{A}$.
\item[(ii)] $\mathbf{A}(s)\mathbf{A}(t)=\mathbf{A}(s+t)$ for all $s,t\in \mathbb{R}$.
\item[(iii)]  $\mathbf{A}'(t)={\cal X}^\Phi(\mathbf{A}(t))$  for all $ t\in \mathbb{R}$.
\end{enumerate}
\end{corollary}

\begin{proof} (i) and (ii) are clear, and (iii) follows from $
\mathbf{A}(t)=\Phi_\mathbf{A}(t)$ being an integral curve for ${\cal X}^\Phi$.
 \end{proof}

We end up pointing   out the similarities  between continuous evolution algebras in matrix Lie groups given by vector fields, see Corollary \ref{COR final},  and the definition of CT-Markov EAs (see \cite[Definition 6.1(i)]{paniello-Markov} or also Section 4). The CT-Markov case corresponds then the flow lines of left and right-invariant vector fields (see Forward and Backward Kolmogorov differential equations) passing  through $\mathbf{A}(0)=\mathbf{I}_n$, and defined on matrix Lie groups with stochastic properties. This last additional condition, together to the finiteness of the state space provides then the nonnegativity of the matrix entries, resulting on Markov structure matrices   for nonnegative time values.

\section{Further comments}

Continuous evolution algebras were firstly proposed in \cite[6.2.4]{Tian EA} as a tool to study the evolution of dynamical systems. Here
we consider them as differentiable curves on  $M_n(\mathbb{K})$ and focus on those algebras arising as solutions of matrix ODEs.    To do this we take    advantage of the differentiability of one-parameter subgroups on matrix Lie groups.

Remark \ref{RM tangent bundle} exhibits the connection between   continuous evolution algebras considered in Section 3 and   tangent bundles of   matrix Lie groups.
Our initial approach is then     generalized by introducing vector fields and  global flows on   matrix Lie groups and considering them as smooth manifolds.
This viewpoint is in fact closer to Lie group (or even, to smooth manifold) theory, as it considers continuous evolution algebras  as flows lines  or, equivalently, orbits of the action of $(\mathbb{R},+)$  on matrix Lie groups.
Indeed,    given $\mathbf{A}(0)=\mathbf{A}$, the set   structure matrices     $\{ \mathbf{A}(t) \mid t\in \mathbb{R}\}$ of  ${\cal E}(t)$ in Corollary~\ref{COR final} consists of the orbit map of $\mathbf{A}$ for the smooth action of $(\mathbb{R},+)$ as Lie group on $G$ determined by the global flow $\Phi$.

The current approach makes possible to bring into a wide  variety of problems, as for instance,
considering
    evolution algebras to model systems whose dynamics obey
 matrix differential equations of the form
   $\mathbf{A}'(t)=\mathbf{A}(t)\mathbf{X}(t)$, with $\mathbf{A}(0)=\mathbf{A}_0$. Equations that,   as far as $\mathbf{X}(t)$ and $\int_0^t  \mathbf{X}(\tau) d\tau$ commute, for all $t$,   admit a solution of the form
$$  \mathbf{A}(t)=\mathbf{A}_0\exp\left(\int_0^t  \mathbf{X}(\tau) d\tau \right).$$
If  the commutativity  requirement fails, then one can   consider Magnus solution $ \mathbf{A}(t)=\mathbf{A}_0\exp(\Omega(t))$ based on the inverse of the derivative of the matrix exponential  \cite[Section~IV.7, Theorem 7.1]{hairer-lubich-wanner}.

 \section*{Declaration of competing interest}

None.


\end{document}